\documentclass[12pt]{amsart}
\usepackage{mathrsfs}

\usepackage{amssymb,amsmath,graphicx,color,textcomp, amsthm,bbm,bbold, enumerate,booktabs}
\usepackage{chngcntr}
\usepackage{apptools}
\AtAppendix{\counterwithin{theorem}{section}}
\definecolor{Red}{cmyk}{0,1,1,0}

\definecolor{verde}{cmyk}{1,0,1,0}

\definecolor{loka}{cmyk}{.5,0,1,.5}
\definecolor{azul}{cmyk}{1,1,0,0}

\numberwithin{equation}{section}

\newcommand{\be}{\begin{equation}}
\newcommand{\ee}{\end{equation}}

\newtheorem{theorem}{Theorem}

\newtheorem{definition}{Definition}
\newtheorem{lemma}{Lemma}
\newtheorem{corollary}[equation]{Corollary}

\newtheorem{example}{Example}

\newtheorem{remark}{Remark}

\begin{document}
\title{On the Ulam-Hyers-Rassias stability for nonlinear fractional	differential equations using the $\psi$-Hilfer operator}
\author{J. Vanterler da C. Sousa$^1$}
\address{$^1$ Department of Applied Mathematics, Institute of Mathematics,
 Statistics and Scientific Computation, University of Campinas --
UNICAMP, rua S\'ergio Buarque de Holanda 651,
13083--859, Campinas SP, Brazil\newline
e-mail: {\itshape \texttt{ra160908@ime.unicamp.br, capelas@ime.unicamp.br }}}

\author{E. Capelas de Oliveira$^1$}

\begin{abstract} 
We study the existence and uniqueness of solution of a nonlinear Cauchy problem involving the $\psi$-Hilfer fractional derivative. In addition, we discuss the Ulam-Hyers and Ulam-Hyers-Rassias stabilities of its solutions.  A few examples are	presented in order to illustrate the possible applications of our main results.

\end{abstract}
\maketitle


\section{Introduction} 

Fractional calculus, or fractional analysis, a branch of mathematical analysis, is a generalization of classical, integer order differentiation and integration to arbitrary, non integer order \cite{AHMJ,SAM,JOSE2,JOSE1}. The discovery of new physical phenomena and the study of chaotic systems have given rise to the proposition of new fractional differential and integral operators that would allow a better description of such systems \cite{RHM,JEM,ATA,FEE,ROSA}. In this context, Sousa and Oliveira \cite{JOSE2} have recently proposed a fractional differentiation operator, which they called $\psi$-Hilfer operator, that has
the special property of unifying several different fractional operators, that is, of generalizing those fractional operators \cite{JOSE2}.

Fractional differential equations arise naturally in different fields such as biology, engineering, medicine, physics and mathematics \cite{AHMJ,SAM,RHM,JEM,ATA,FEE,ROSA}. Using fractional derivatives to model
phenomena has proved an excellent tool. On the other hand, this also intensifies the studies about the existence and uniqueness of solutions of fractional differential equations with fractional differentiation operators, either with time delay, linear or nonlinear. We refer the reader, for example, to references \cite{EXIS,EXIS1,EXIS2,EXIS3,EXIS5,EXIS6}.

The stability problem of differential equations was formulated and discussed by Ulam and Hyers \cite{EST31,EST32,EST13,EST4}.  Between 1978 and 1988, Rassias established the Ulam-Hyers stability of linear and nonlinear mappings \cite{EST,EST1}. The study of Ulam stability and data dependence of fractional
differential equations was initiated by Wang et al. \cite{ESTF1}. Studies about Ulam-Hyers and Ulam-Hyers-Rassias stability for fractional differential equations can be found in \cite{ESTF1,ESTF,ESTF2,ESTF3}.

In this paper, we consider nonlinear fractional differential equations
\begin{equation}\label{CPT}
\left\{ 
\begin{array}{rcl}
^{H}\mathbb{D}_{a+}^{\alpha ,\beta ;\psi }y\left( t\right)  & = & f\left( t,y\left(
t\right), \text{ }^{H}\mathbb{D}_{a+}^{\alpha ,\beta ;\psi }y\left( t\right) \right)  \\ 
I_{a+}^{1-\gamma ;\psi }y\left( a\right)  & = & y_{a}
\end{array}
\right. 
\end{equation}
where $^{H}\mathbb{D}_{a+}^{\alpha ,\beta ;\psi }(\cdot)$ is the $\psi$-Hilfer fractional derivative \cite{JOSE2} of order $0<\alpha \leq 1$ and type $0\leq \beta \leq 1$, $I^{1-\gamma;\psi}_{a+}(\cdot)$ is the Riemann-Liouville fractional integral of order $1-\gamma$, $\gamma=\alpha+\beta(1-\alpha)$, with respect to function $\psi$ \cite{AHMJ,SAM}, $f:J\times \mathbb{R} \times \mathbb{R} \rightarrow  \mathbb{R} $ is a given function space, $t\in J=\left[ a,T\right]$ with $T>a$ and $y_{a}\in\mathbb{R}$.

The main purpose of this paper is to study the existence and uniqueness of solutions for the nonlinear Cauchy problem Eq.(\ref{CPT}), by means of Banach's contraction principle. In addition, we demonstrate four types
of stability, Ulam-Hyers, generalized Ulam-Hyers, Ulam-Hyers-Rassias and generalized Ulam-Hyers-Rassias stabilities for the fractional differential equation Eq.(\ref{CPT}) in the case $0<\alpha<1$.

The paper is organized as follows: In section 2, we present the one parameter Mittag-Leffler function and some particular cases. We present the definitions of the $\psi$-Hilfer fractional derivative and the Riemann-Liouville fractional integral with respect to a function $\psi$, the spaces in which these fractional operators are defined and Gronwall's inequality, among other results \cite{JOSE}. We also present four definitions of stability Ulam-Hyers, generalized Ulam-Hyers, Ulam-Hyers-Rassias and generalized Ulam-Hyers-Rassias using the $\psi$-Hilfer fractional derivative. In section 3, we study the existence and uniqueness of the solutions of the proposed nonlinear Cauchy problem, using Banach's contraction principle. In section 4, we study the stabilities presented in section 2 and make some important observations about them. In section 5, we present some particular cases of the Cauchy problem we have been studying and discuss their stability according to the criteria of Ulam-Hyers and Ulam-Hyers-Rassias. We also discuss the importance and advantages of using the $\psi$-Hilfer fractional derivative. Concluding remarks close the paper.


\section{Preliminaries}

In this section, we present the definition of the one parameter Mittag-Leffler function and some particular cases. We introduce the definitions of the $\psi$-Hilfer fractional derivative and the Riemann-Liouville fractional integral with respect to a function, the spaces in which they are defined and theorems involving these operators, in particular, Gronwall's inequality. One of the main results of this paper is the study of the stability of the solutions Cauchy's problem, Eq.(\ref{CPT}). We did this for the definitions of Ulam-Hyers, generalized Ulam-Hyers, Ulam-Hyers-Rassias and generalized Ulam-Hyers-Rassias stabilities using the $\psi$-Hilfer fractional derivative.

The classical Mittag-Leffler function is the most important function of fractional calculus, specially in the study of linear fractional differential equations with constant coefficients. It also plays an important role in the study of the stability of the solutions of linear and nonlinear differential equations.  We will deal only with the one parameter Mittag-Leffler function; for Mittag-Leffler functions with two, three and more parameters we suggest the book \cite{GKAM}.

In 1903, Mittag-Leffler \cite{ML1} introduced the classic Mittag-Leffler function with only one complex parameter.

\begin{definition}{\rm \cite{GKAM}}
\label{def1}\rm{(One parameter Mittag-Leffler function)}. The Mittag-Leffler function is given by the series 
\begin{equation} \label{A1}
\mathbb{E}_{\mu }\left( z\right) =\overset{\infty }{\underset{k=0}{\sum }}
\frac{z^{k}}{\Gamma \left( \mu k+1\right) },  
\end{equation}
where $\mu \in \mathbb{C}$, ${Re}\left( \mu \right) >0$ and $\Gamma (z)$ is a
gamma function, given by 
\begin{equation*}
\Gamma \left( z\right) =\int_{0}^{\infty }e^{-t}t^{z-1}dt,
\end{equation*}
$\mbox{Re}\left( z\right) >0$.
\end{definition}

The error function is defined by means of the following integral: 
\begin{equation*}
\text{erf}\left(z\right) =\dfrac{2}{\sqrt{\pi}}\int_{0}^{z}e^{-t^{2}}dt.
\end{equation*}

In particular, if $\mu =1/2$ in Eq.(\ref{A1}), we have 
\begin{eqnarray}\label{A2}
\mathbb{E}_{1/2}\left( z\right)  &=&\overset{\infty }{\underset{k=0}{\sum }}\frac{
z^{k}}{\Gamma \left( \frac{k}{2}+1\right) }=\frac{1}{\Gamma \left( 1\right) }+
\frac{z}{\Gamma \left( \frac{3}{2}\right) }+\frac{z^{2}}{\Gamma \left(
2\right) }+\frac{z^{3}}{\Gamma \left( \frac{5}{2}\right) }+ \\ \notag
&&+\frac{z^{4}}{\Gamma \left( 3\right) }+\frac{z^{5}}{\Gamma \left( \frac{7}{
2}\right) }+\cdot \cdot \cdot +\frac{\left( z^{2}\right) ^{n}}{n!}+\frac{
z^{2n+1}}{\Gamma \left( n+\frac{3}{2}\right) }+\cdot \cdot \cdot 
\end{eqnarray}

Note that, for $k=2n,$ $n\in \mathbb{N},$ we have
\begin{equation}\label{A3}
A\left( z\right) =\overset{\infty }{\underset{k=0}{\sum }}\frac{\left(
z^{2}\right) ^{k}}{k!}=\exp \left( z^{2}\right) .
\end{equation}

On the other hand, for $k=2n+1,$ $n\in \mathbb{N},$ we have
\begin{equation}\label{A4}
B\left( z\right) =\overset{\infty }{\underset{k=0}{\sum }}\frac{z^{2n+1}}{
\Gamma \left( n+\frac{3}{2}\right) }=\exp \left( z^{2}\right) \text{erf}\left( z\right) .
\end{equation}

Then, substituting Eq.(\ref{A3}) and Eq.(\ref{A4}) into Eq.(\ref{A2}), we have
\begin{equation}\label{A5}
\mathbb{E}_{1/2}\left( z\right) =\overset{\infty }{\underset{k=0}{\sum }}\frac{\left(
z^{2}\right) ^{k}}{k!}+\overset{\infty }{\underset{k=0}{\sum }}\frac{z^{2n+1}
}{\Gamma \left( n+\frac{3}{2}\right) }=\exp \left( z^{2}\right) [ 1+\text{erf}\left( z\right)].
\end{equation}

For $z=\lambda ^{\beta },$ in Eq.(\ref{A5}), we have
\begin{equation}\label{A6}
\mathbb{E}_{1/2}\left( \lambda ^{\beta }\right) =\exp \left( \lambda ^{2\beta
}\right) [1+\text{erf}\left( \lambda ^{\beta }\right)].
\end{equation}

Taking the limits $\beta \rightarrow 1$ and $\beta \rightarrow 0$, on both sides of Eq.(\ref{A6}), we have
\begin{equation*}
\mathbb{E}_{1/2}\left( \lambda \right) =\exp \left( \lambda ^{2}\right) [1+\text{erf}\left( \lambda \right)]
\end{equation*}
and
\begin{equation}\label{A7}
\mathbb{E}_{1/2}\left( 1\right) =\exp \left( 1\right) [1+\text{erf}\left(1\right)] \simeq 5.002,
\end{equation}
respectively.

Let $[a,b]$ $(0<a<b<\infty)$ be a finite interval on the half-axis $\mathbb{R}^{+}$ and let $C[a,b]$ be the space of continuous functions $f$ on $[a,b]$ with the norm defined by
\cite{AHMJ}
\begin{equation*}
\left\Vert f\right\Vert _{C\left[ a,b\right] }=\underset{t\in \left[ a,b \right] }{\max }\left\vert f\left( t\right) \right\vert.
\end{equation*}

The weighted space $C_{1-\gamma ;\psi }\left( \left[ a,b\right],\mathbb{R}\right) $ of functions $f$ on $(a,b]$ is defined by
\begin{equation*}
C_{1-\gamma ;\psi }\left[ a,b\right] =\left\{ f:\left( a,b\right] \rightarrow 
\mathbb{R};\text{ }\left( \psi \left( t\right) -\psi \left( a\right) \right)
^{1-\gamma }f\left( t\right) \in C\left[ a,b\right] \right\} ,\text{ }0\leq \gamma <1 , 
\end{equation*}
with the norm
\begin{equation*}
\left\Vert f\right\Vert _{C_{1-\gamma ;\psi }\left[ a,b\right] }=\left\Vert
\left( \psi \left( t\right) -\psi \left( a\right) \right) ^{1-\gamma}f\left(
t\right) \right\Vert _{C\left[ a,b\right] }=\underset{t\in \left[ a,b\right] 
}{\max }\left\vert \left( \psi \left( t\right) -\psi \left( a\right) \right)
^{1-\gamma }f\left( t\right) \right\vert.
\end{equation*}

The weighted space $C_{\gamma ;\psi }^{n}\left[ a,b\right]$ of functions $f$ on $(a,b]$ is defined by
\begin{equation*}
C_{\gamma;\psi }^{n}\left[ a,b\right] =\left\{ f:\left( a,b\right]
\rightarrow \mathbb{R};\text{ }f\left( t\right) \in C^{n-1}\left[ a,b\right] ;\text{ }f^{\left( n\right) }\left( t\right) \in C_{\gamma;\psi }\left[ a,b\right] \right\} ,\text{ }0\leq \gamma <1 ,
\end{equation*}
with the norm
\begin{equation*}
\left\Vert f\right\Vert _{C_{\gamma ;\psi }^{n}\left[ a,b\right] }=\overset{n-1}{\underset{k=0}{\sum }}\left\Vert f^{\left( k\right) }\right\Vert _{C\left[ a,b\right] }+\left\Vert f^{\left( n\right) }\right\Vert _{C_{\gamma ;\psi }\left[ a,b\right] }.
\end{equation*}

For $n=0$, we have, $C_{\gamma,\psi }^{0}\left[ a,b\right] =C_{\gamma,\psi }\left[ a,b\right] $.

The weighted space $C^{\alpha,\beta}_{\gamma,\psi}[a,b]$ is defined by
\begin{equation*}
C_{\gamma ;\psi }^{\alpha ,\beta }\left[ a,b\right] =\left\{ f\in C_{\gamma;\psi }\left[ a,b\right] ;\text{ }^{H}\mathbb{D}_{a+}^{\alpha ,\beta ;\psi }f\in C_{\gamma;\psi }\left[ a,b\right] \right\} ,\text{ }\gamma =\alpha +\beta \left( 1-\alpha\right) .
\end{equation*}

We now present the Riemann-Liouville fractional integral with respect to a function $\psi$ and the $\psi$-Hilfer fractional derivative, recently introduced by Sousa and Oliveira \cite{JOSE2}.

\begin{definition}{\rm \cite{AHMJ}} Let $(a,b)$ $(-\infty \leq a<b \leq \infty)$ be a finite or infinite interval of the real line $\mathbb{R}$ and let $\alpha>0$. Also let $\psi(x)$ be an increasing and positive monotone function on $(a,b]$, having a continuous derivative $\psi'(t)$ on $(a,b)$. The left-sided fractional
integral of a function $f$ with respect to function $\psi$ on $[a,b]$ is defined by
\begin{equation}\label{A}
I_{a+}^{\alpha ;\psi }f\left( t\right) =\frac{1}{\Gamma \left( \alpha
\right) }\int_{a}^{t}\psi ^{\prime }\left( t\right) \left( \psi \left(
t\right) -\psi \left( s\right) \right) ^{\alpha -1}f\left( t\right) ds.
\end{equation}
The right-sided fractional integral is defined in an analogous form {\rm \cite{AHMJ}}.

\end{definition}

\begin{definition}{\rm \cite{JOSE2}} Let $n-1<\alpha <n$ with $n\in\mathbb{N}$; let $I=[a,b]$ be an interval such that $-\infty\leq a<b\leq\infty$ and let $f,\psi\in C^{n}([a,b],\mathbb{R})$ be two functions such that $\psi$ is	increasing and $\psi'(t)\neq 0$, for all $t\in I$. The left-sided $\psi$-Hilfer fractional derivative $^{H}\mathbb{D}_{a+}^{\alpha ,\beta ;\psi }\left( \cdot\right)$ of function $f$, of order $\alpha$ and type $0\leq \beta \leq 1$, is defined by
\begin{equation}\label{HIL}
^{H}\mathbb{D}_{a+}^{\alpha ,\beta ;\psi }f\left( t\right) =I_{a+}^{\beta \left(
n-\alpha \right) ;\psi }\left( \frac{1}{\psi ^{\prime }\left( t\right) }\frac{d}{dt}\right) ^{n}I_{a+}^{\left( 1-\beta \right) \left( n-\alpha \right) ;\psi }f\left( t\right).
\end{equation}

The right-sided $\psi$-Hilfer fractional derivative is defined in an analogous form {\rm \cite{JOSE2}}.

\end{definition}

The $\psi$-Hilfer fractional derivative, as above defined, can be written in
the form

\begin{equation}\label{HIL2}
^{H}\mathbb{D}_{a+}^{\alpha ,\beta ;\psi }f\left( t\right) =I_{a+}^{\gamma -\alpha ;\psi }\mathcal{D}_{a+}^{\gamma ;\psi }f\left( t\right) , 
\end{equation}
with $\gamma =\alpha +\beta \left( n-\alpha \right) $ and where $I^{\gamma-\alpha;\psi}_{a+}(\cdot)$ is the $\psi$-Riemann-Liouville fractional integral and $\mathcal{D}^{\gamma;\psi}_{a+}(\cdot)$ is the $\psi$-Riemann-Liouville derivative \cite{AHMJ,SAM}.  Here we consider the $\psi$-Hilfer fractional derivative for $n=1$.

The following two theorems, will be important throughout the paper.

\begin{theorem}\label{teo1} If $f\in C_{\gamma ,\psi }^{1}[a,b]$, $0<\alpha <1$ and $0\leq \beta \leq 1$, then 
\begin{equation*}
I_{a+}^{\alpha ;\psi }\text{ }^{H}\mathbb{D}_{a+}^{\alpha ,\beta ;\psi
}f\left( t\right) =f\left( t\right) -\frac{\left( \psi \left( x\right) -\psi
\left( a\right) \right) ^{\gamma -1}}{\Gamma \left( \gamma \right) }%
I_{a+}^{\left( 1-\beta \right) \left( 1-\alpha \right) ;\psi }f\left(
a\right) .
\end{equation*}
\end{theorem}

\begin{proof}
See {\rm \cite{JOSE2}}.
\end{proof}

\begin{theorem}\label{teo2}
	Let $f\in C^{1}_{\gamma,\psi}[a,b]$, $\alpha>0$ and $0\leq\beta\leq 1$; then we have
\begin{equation*}
^{H}\mathbb{D}_{a+}^{\alpha ,\beta ;\psi }I_{a+}^{\alpha ;\psi }f\left( t\right) =f\left( t\right).
\end{equation*}
\end{theorem}

\begin{proof}
See {\rm \cite{JOSE2}}.
\end{proof}

Recently, Sousa and Oliveira \cite{JOSE2} introduced a new class of fractional derivatives and integrals and, using a result on Gronwall's inequality \cite{JOSE}, obtained the corresponding class of Gronwall's inequalities. 

Their results allow us to state the following Gronwall's inequality involving the $\psi$-Riemann-Liouville fractional integral. 

\begin{theorem}\label{teo3} Let $u,$ $v$ be two integrable functions and $g$ a 	continuous function, with domain $\left[ a,b\right] .$ Let $\psi \in C^{1}\left[ a,b\right] $ an increasing function such that $\psi
^{\prime }\left( t\right) \neq 0$, $\forall t\in \left[ a,b\right]$. Assume that
\begin{enumerate}
\item $u$ and $v$ are nonnegative;
\item $g$ in nonnegative and nondecreasing.
\end{enumerate}
If
\begin{equation*}
u\left( t\right) \leq v\left( t\right) +g\left( t\right) \int_{a}^{t}\psi ^{\prime }\left( s\right) \left( \psi \left( t\right) -\psi \left(s \right) \right) ^{\alpha -1}u\left( s \right) ds,
\end{equation*}
then
\begin{equation}\label{jose}
u\left( t\right) \leq v\left( t\right) +\int_{a}^{t}\overset{\infty }{\underset{k=1}{\sum }}\frac{\left[ g\left( t\right) \Gamma \left( \alpha \right) \right] ^{k}}{\Gamma \left( \alpha k\right) }\psi ^{\prime }\left(
s \right) \left[ \psi \left( t\right) -\psi \left( s\right) \right]^{\alpha k-1}v\left( s\right) ds,
\end{equation}
$\forall t\in \left[ a,b\right]$.
\end{theorem}

\begin{proof}
See {\rm \cite{JOSE}}.
\end{proof}

\begin{corollary}{\rm \cite{JOSE}} Let $\alpha>0$, $I=[a,b]$ and $\psi\in C^{1}([a,b],\mathbb{R})$ a function such that $\psi$ is increasing and $\psi'(t)\neq 0$ for all $t\in I$. Suppose that $b\geq 0$, $v$ is a nonnegative function locally integrable on $\left[ a,b\right]$ and $u$ is nonnegative and locally integrable on $\left[a,b\right] $ with 
\begin{equation*}
u\left( t\right) \leq v\left( t\right) +b\int_{a}^{t}\psi ^{\prime }\left(s \right) \left[ \psi \left( t\right) -\psi \left( s\right) \right] ^{\alpha -1}u\left(s\right) ds, \text{ }\forall t\in[a,b].
\end{equation*}

Then, we can write
\begin{equation*}
u\left( t\right) \leq v\left( t\right) +\int_{a}^{t}\underset{k=1}{\overset{\infty }{\sum }}\frac{\left[ b\Gamma \left( \alpha \right) \right] ^{k}}{\Gamma \left( \alpha k\right) }\psi ^{\prime }\left(s\right) \left[
\psi \left( t\right) -\psi \left(s\right) \right] ^{\alpha k-1}v\left(s\right) ds, \text{ }\forall t\in[a,b].
\end{equation*}
\end{corollary}

\begin{corollary} Under the hypotheses of {\rm Theorem \ref{teo3}},	let $v$ be a nondecreasing function on $\left[ a,b\right] $. Then, we have
\begin{equation*}
u\left( t\right) \leq v\left( t\right) \mathbb{E}_{\alpha }\left( g\left( t\right) \Gamma \left( \alpha \right) \left[ \psi \left( t\right) -\psi \left(s\right) \right] ^{\alpha }\right) ,\text{ }\forall t\in \left[ a,b\right],
\end{equation*}
where $\mathbb{E}_{\alpha }\left( \cdot \right) $ is the Mittag-Leffler function defined by {\rm Eq.(\ref{A1})}.
\end{corollary}

\begin{proof} 
See {\rm \cite{JOSE}}.
\end{proof}

In the theory of fractional differential equations there are several types of Ulam stability. We present here the Ulam-Hyers stability, the Ulam-Hyers-Rassias stability and their respective generalizations using the $\psi$-Hilfer fractional derivative.

\begin{definition} The {\rm Eq.(\ref{CPT})} is Ulam-Hyers stable if there exists a real number $c_{f}>0$ such that for each $\varepsilon>0$ and for each solution $z\in C^{1}_{1-\gamma,\psi}(J,\mathbb{R})$ of the inequality
\begin{equation}\label{CPT1}
\left\vert ^{H}\mathbb{D}_{a+}^{\alpha ,\beta ;\psi }z\left( t\right) -f\left(
t,z\left( t\right) ,^{H}\mathbb{D}_{a+}^{\alpha ,\beta ;\psi }z\left( t\right)
\right) \right\vert \leq \varepsilon, \text{ } t\in J
\end{equation}
there exists a solution $y\in C^{1}_{1-\gamma,\psi}(J,\mathbb{R})$ of
{\rm Eq.(\ref{CPT})}, such that 
\begin{equation*}
\left\vert z\left( t\right) -y\left( t\right) \right\vert \leq
c_{f}\varepsilon, \text{ } t\in J.
\end{equation*}
\end{definition}

\begin{definition} The {\rm Eq.(\ref{CPT})} is generalized Ulam-Hyers stable if there exists $\Phi_{f}\in C(\mathbb{R_{+}},\mathbb{R_{+}})$, $\Phi_{f}(0)=0$, such that for each solution $z\in	C^{1}_{1-\gamma,\psi}(J,\mathbb{R})$ of the inequality {\rm	Eq.(\ref{CPT1})} there exists a solution $y\in C^{1}_{1-\gamma,\psi}(J,\mathbb{R})$ of {\rm Eq.(\ref{CPT})} such that
\begin{equation*}
\left\vert z\left( t\right) -y\left( t\right) \right\vert \leq
\Phi_{f}\varepsilon, \text{ } t\in J.
\end{equation*}
\end{definition}

\begin{definition} The {\rm Eq.(\ref{CPT})} is Ulam-Hyers-Rassias stable with respect to $\varphi\in C(J,\mathbb{R})$, if there exists a real number $c_{f}>0$ such that for each $\varepsilon>0$ and for each solution $z\in C^{1}_{1-\gamma,\psi}(J,\mathbb{R})$ of the inequality
\begin{equation}\label{CPT2}
\left\vert ^{H}\mathbb{D}_{a+}^{\alpha ,\beta ;\psi }z\left( t\right) -f\left(
t,z\left( t\right) ,^{H}\mathbb{D}_{a+}^{\alpha ,\beta ;\psi }z\left( t\right)
\right) \right\vert \leq \varepsilon \varphi(t), \text{ } t\in J
\end{equation}
there exists a solution $y\in C^{1}_{1-\gamma,\psi}(J,\mathbb{R})$ of {\rm Eq.(\ref{CPT})}, 
such that 
\begin{equation*}
\left\vert z\left( t\right) -y\left( t\right) \right\vert \leq
c_{f}\varepsilon \varphi(t), \text{ } t\in J.
\end{equation*}
\end{definition}

\begin{definition} The {\rm Eq.(\ref{CPT})} is generalized Ulam-Hyers-Rassias stable with respect to $\varphi\in C(J,\mathbb{R_{+}})$, if there exists a real number $c_{f,\varphi}>0$ such that for each solution $z\in C^{1}_{1-\gamma,\psi}(J,\mathbb{R})$ of the inequality
\begin{equation}\label{CPT3}
\left\vert ^{H}\mathbb{D}_{a+}^{\alpha ,\beta ;\psi }z\left( t\right) -f\left(
t,z\left( t\right) ,^{H}\mathbb{D}_{a+}^{\alpha ,\beta ;\psi }z\left( t\right)
\right) \right\vert \leq \varphi(t), \text{ } t\in J,
\end{equation}
there exists a solution $y\in C^{1}_{1-\gamma,\psi}(J,\mathbb{R})$ of
{\rm Eq.(\ref{CPT})} such that 
\begin{equation*}
\left\vert z\left( t\right) -y\left( t\right) \right\vert \leq
c_{f,\varphi}\varphi(t), \text{ } t\in J.
\end{equation*}
\end{definition}

\begin{remark}
A function $z\in C^{1}_{1-\gamma,\psi}(J,\mathbb{R})$ is a solution of the inequality {\rm Eq.(\ref{CPT1})}  if and only if there exists a function $g\in C(J,\mathbb{R})$ such that
\begin{itemize}
\item $\left\vert g\left( t\right) \right\vert \leq \varepsilon $.
\item $^{H}\mathbb{D}_{a+}^{\alpha ,\beta ;\psi }z\left( t\right) =f\left( t,z\left(
t\right) ,^{H}\mathbb{D}_{a+}^{\alpha ,\beta ;\psi }z\left( t\right) \right) +g\left(
t\right) ,\text{ }t\in J$.
\end{itemize}
\end{remark}

\section{Existence and uniqueness of solutions to $\psi$-Hilfer nonlinear
fractional differential equations} 

The nonlinear Cauchy-type problem using fractional derivatives has been shown to be of great interest to the academic community, not only mathematicians but also researchers in other fields. The study of the existence and uniqueness of solutions of Cauchy-type problem has an interesting way of obtaining this result, namely, through the integral equation. In this section, we present and prove a lemma which guarantees the equivalence between the Volterra integral equation and the nonlinear Cauchy problem Eq.(\ref{CPT}). In the sequence, using Banach's contraction principle, we present the proof of the first main result of this paper, the existence and uniqueness of solutions to the nonlinear Cauchy problem.

\begin{lemma}\label{lema1} Let a function $f\left( t,u,v\right) :J\times \mathbb{R} \times \mathbb{R}\rightarrow \mathbb{R}$ be continuous. Then the problem {\rm Eq.(\ref{CPT})} is equivalent to the problem
\begin{equation}\label{casa}
y\left( t\right) =\frac{\left( \psi \left( t\right) -\psi \left( a\right)
\right) ^{\gamma -1}}{\Gamma \left( \gamma \right) }y_{a}+I_{a+}^{\alpha
;\psi }g\left( t\right)
\end{equation}
where $g\in C\left( J,\mathbb{R}\right) $ satisfies the functional equation
\begin{equation*}
g\left( t\right) =f\left( t,\frac{\left( \psi \left( t\right) -\psi \left(
a\right) \right) ^{\gamma -1}}{\Gamma \left( \gamma \right) }%
y_{a}+I_{a+}^{\alpha ;\psi }g\left( t\right) ,g\left( t\right) \right).
\end{equation*}
\end{lemma}

\begin{proof} Applying the fractional integral operator $I_{a+}^{\alpha ;\psi }\left( \cdot \right) $ on both sides of the fractional equation Eq.(\ref{CPT}) and using Theorem \ref{teo1}, we get
\begin{equation*}
y\left( t\right) -\frac{\left( \psi \left( t\right) -\psi \left( a\right) \right) ^{\gamma -1}}{\Gamma \left( \gamma \right) }I_{a+}^{\left( 1-\beta \right) \left( 1-\alpha \right) ;\psi }y\left( a\right) =I_{a+}^{\alpha
;\psi }g\left( t\right). 
\end{equation*}

Therefore,
\begin{equation}\label{casa1}
y\left( t\right) =\frac{\left( \psi \left( t\right) -\psi \left( a\right)\right) ^{\gamma -1}}{\Gamma \left( \gamma \right) }y_{a}+I_{a+}^{\alpha;\psi }g\left( t\right) .
\end{equation}

On the other hand, if $y$ satisfies Eq.(\ref{casa1}), then it satisfies Eq.(\ref{CPT}). However, applying the fractional derivative	$^{H}\mathbb{D}_{a+}^{\alpha ,\beta ;\psi }\left( \cdot \right) $ on both sides of the Eq.(\ref{CPT}) and using  Theorem \ref{teo2}, we	have
\begin{eqnarray*}
^{H}\mathbb{D}_{a+}^{\alpha ,\beta ;\psi }y\left( t\right)  &=&^{H}\mathbb{D}_{a+}^{\alpha,\beta ;\psi }\frac{\left( \psi \left( t\right) -\psi \left( a\right)\right) ^{\gamma -1}}{\Gamma \left( \gamma \right) }y_{a}+\text{ }^{H}\mathbb{D}_{a+}^{\alpha ,\beta ;\psi }I_{a+}^{\alpha ;\psi }g\left( t\right)  \\&=&g\left( t\right),
\end{eqnarray*}
where, for $0<\gamma <1$, we have $^{H}\mathbb{D}_{a+}^{\alpha ,\beta ;\psi
	}\left(\psi \left( t\right) -\psi \left( a\right) \right) ^{\gamma
	-1}=0$ \cite{JOSE2,EXIS}.
\end{proof}

\begin{theorem}\label{teo4} We assume the following hypotheses:
\begin{enumerate}
\item{\rm (H1)} The function $f:J\times \mathbb{R}\times \mathbb{R}\rightarrow \mathbb{R}$ is continuous.
\item {\rm (H2)} There exist constants $k>0$ and $l>0$ such that
\begin{equation*}
\left\vert f\left( t,u,v\right) -f\left( \overline{t},\overline{u},\overline{v}\right) \right\vert \leq k\left\vert u-\overline{u}\right\vert+l\left\vert v-\overline{v}\right\vert 
\end{equation*}
for any $u,v,\overline{u},\overline{v}\in \mathbb{R}$ and $t\in J$. 
\end{enumerate}
If
\begin{equation}\label{casa2}
\left( \frac{k}{\Gamma \left( \alpha +1\right) }\left( \psi \left( t\right)
-\psi \left( a\right) \right) ^{\alpha }+l\right) <1
\end{equation}
then there exists a unique solution for the Cauchy problem Eq.(\ref{CPT}) on $J$.
\end{theorem}

\begin{proof} Define the operator $M:C_{1-\gamma ,\psi }\left( J,\mathbb{R}\right) \rightarrow C_{1-\gamma ,\psi }\left( J,\mathbb{R} \right) $ by
\begin{equation}\label{casa3}
Mz\left( t\right) =f\left( t,\frac{\left( \psi \left( t\right) -\psi \left( a\right) \right) ^{\gamma -1}}{\Gamma \left( \gamma \right) }y_{a}+\frac{1}{\Gamma \left( \alpha \right) }\int_{a}^{t}\psi ^{\prime }\left( s\right)
\left( \psi \left( t\right) -\psi \left( s\right) \right) ^{\alpha
-1}z\left( s\right) ds,z\left( s\right) \right) 
\end{equation}
for each $t\in J.$

Let $u,w\in C_{1-\gamma ,\psi }\left( J,\mathbb{R} \right)$. Then for $t\in J=\left[ a,T\right] $ and using hypothesis (H1), we have
\begin{eqnarray}\label{casa4}
&&\left\vert \left( \psi \left( t\right) -\psi \left( a\right) \right)
^{1-\gamma }\left( Mu\left( t\right) -Mw\left( t\right) \right) \right\vert 
\nonumber \\
&\leq &\frac{k}{\Gamma \left( \alpha \right) }\int_{a}^{t}\psi ^{\prime
}\left( s\right) \left( \psi \left( t\right) -\psi \left( s\right) \right)
^{\alpha -1}\left\vert \left( \psi \left( t\right) -\psi \left( a\right)
\right) ^{1-\gamma }\left( u\left( s\right) -w\left( s\right) \right)
\right\vert ds  \nonumber \\
&&+l\left\vert \left( \psi \left( t\right) -\psi \left( a\right) \right)
^{1-\gamma }\left( u\left( t\right) -w\left( t\right) \right) \right\vert  
\nonumber \\
&\leq &\frac{k\left\Vert u-w\right\Vert _{C_{1-\gamma ,\psi }}}{\Gamma
\left( \alpha \right) }\int_{a}^{t}\psi ^{\prime }\left( s\right) \left(
\psi \left( t\right) -\psi \left( s\right) \right) ^{\alpha -1}ds+l\left\Vert
u-w\right\Vert _{C_{1-\gamma ,\psi }}  \nonumber \\
&\leq &\left[ \frac{k}{\Gamma \left( \alpha +1\right) }\left( \psi \left(
T\right) -\psi \left( a\right) \right) ^{\alpha }+l\right] \left\Vert
u-w\right\Vert _{C_{1-\gamma ,\psi }}.
\end{eqnarray}

Evaluating the maximum values, for $t\in \left[ a,b\right]$, of both sides of Eq.(\ref{casa4}) and using the definition of the norm in the weighted	space, we get 
\begin{equation*}
\left\Vert Mu-Mw\right\Vert _{C_{1-\gamma ,\psi }}\leq \left[ \frac{k}{\Gamma \left( \alpha +1\right) }\left( \psi \left( T\right) -\psi \left(a\right) \right) ^{\alpha }+l\right] \left\Vert u-w\right\Vert _{C_{1-\gamma,\psi }}.
\end{equation*}

It follows from this result and Eq.(\ref{casa2}) that operator $M$ is a	contraction, then, we conclude that operator $M$ has a unique fixed	point $z\in C_{1-\gamma,\psi }\left( J,\mathbb{R}\right) $, given by
Banach's contraction principle.

Therefore,
\begin{equation*}
z\left( t\right) =f\left( t,y\left( t\right) ,z\left( s\right) \right) ,
\end{equation*}
for each $t\in J$, where
\begin{equation*}
y\left( t\right) =\frac{\left( \psi \left( t\right) -\psi \left( a\right) \right) ^{\gamma -1}}{\Gamma \left( \gamma \right) }y_{a}+\frac{1}{\Gamma \left( \alpha \right) }\int_{a}^{t}\psi ^{\prime }\left( s\right) \left( \psi \left( t\right) -\psi \left( s\right) \right) ^{\alpha -1}z\left( s\right) ds.
\end{equation*}

This implies that $^{H}\mathbb{D}_{a+}^{\alpha ,\beta ;\psi }y\left(
	t\right)=z\left( t\right) $. Consequently, 
\begin{equation*}
^{H}\mathbb{D}_{a+}^{\alpha ,\beta ;\psi }y\left( t\right) =f\left( t,y\left(t\right),^{H}\mathbb{D}_{a+}^{\alpha ,\beta ;\psi }y\left( t\right) \right) .
\end{equation*}
\end{proof}

\section{Ulam-Hyers and Ulam-Hyers-Rassias stabilities}

In this section, we present and prove two theorems showing that the nonlinear Cauchy problem Eq.(\ref{CPT}) using $\psi$-Hilfer fractional derivative admits Ulam-Hyers and Ulam-Hyers-Rassias stabilities.

The first theorem is about Ulam-Hyers stability.

\begin{theorem}\label{teo5} Suppose the validity of {\rm (H1)},
{\rm (H2)} and {\rm Eq.(\ref{casa2})}. Then {\rm Eq.(\ref{CPT})} is Ulam-Hyers stable.
\end{theorem}

\begin{proof}
Let $z\in C_{1-\gamma ,\psi }\left( J,\mathbb{R}\right) $ be a solution of 
inequality {\rm Eq.(\ref{CPT1})}, i.e.,
\begin{equation}\label{casa5}
\left\vert ^{H}\mathbb{D}_{a+}^{\alpha ,\beta ;\psi }z\left( t\right) -f\left(
t,z\left( t\right) ,^{H}\mathbb{D}_{a+}^{\alpha ,\beta ;\psi }z\left( t\right)
\right) \right\vert \leq \varepsilon ,\text{ }t\in J.
\end{equation}

	Let us denote by $y\in C_{1-\gamma ,\psi }\left( J,\mathbb{R}\right) $ the unique solution of the Cauchy problem, so that 
\begin{equation*}
^{H}\mathbb{D}_{a+}^{\alpha ,\beta ;\psi }y\left( t\right) =f\left( t,y\left(t\right) ,^{H}\mathbb{D}_{a+}^{\alpha ,\beta ;\psi }y\left( t\right) \right)
\end{equation*}
for each $t\in J,$ $0<\alpha \leq 1$ and $\ 0\leq \beta \leq 1;$ $I_{a+}^{\left( 1-\beta \right) \left( 1-\alpha \right) ;\psi }y\left(a\right) =I_{a+}^{\left( 1-\beta \right) \left( 1-\alpha \right) ;\psi}z\left( a\right)$.

Using Lemma \ref{lema1}, we have
\begin{equation*}
y\left( t\right) =\frac{\left( \psi \left( t\right) -\psi \left( a\right) \right) ^{\gamma -1}}{\Gamma \left( \gamma \right) }y_{a}+\frac{1}{\Gamma \left( \alpha \right) }\int_{a}^{t}\psi ^{\prime }\left( s\right) \left(
\psi \left( t\right) -\psi \left( s\right) \right) ^{\alpha -1}g_{y}\left( s\right) ds,
\end{equation*}
where $g_{y}\in C_{1-\gamma ,\psi }\left( J,\mathbb{R}\right)$ satisfies the functional equation
\begin{equation*}
g_{y}\left( t\right) =f\left( t,\frac{\left( \psi \left( t\right) -\psi\left( a\right) \right) ^{\gamma -1}}{\Gamma \left( \gamma \right) }y_{a}+I_{a+}^{\alpha ;\psi }g_{y}\left( t\right) ,g_{y}\left( t\right)\right) .
\end{equation*}

Applying operator $I_{a+}^{\alpha ;\psi }\left( \cdot \right) $ on both sides of Eq.(\ref{casa5}) and using Theorem \ref{teo1}, we have
\begin{equation*}
\left\vert I_{a+}^{\alpha ;\psi }\text{ }^{H}\mathbb{D}_{a+}^{\alpha ,\beta ;\psi }z\left( t\right) -I_{a+}^{\alpha ;\psi }\text{ }^{H}f\left( t,z\left( t\right), ^{H}\mathbb{D}_{a+}^{\alpha ,\beta ;\psi }z\left( t\right) \right)
\right\vert \leq I_{a+}^{\alpha ;\psi }\varepsilon;
\end{equation*}
this implies that 
\begin{eqnarray*}
&&\left\vert z\left( t\right) -\frac{\left( \psi \left( t\right) -\psi \left( a\right) \right) ^{\gamma -1}}{\Gamma \left( \gamma \right) }z_{a}-\frac{1}{ \Gamma \left( \alpha \right) }\int_{a}^{t}\psi ^{\prime }\left( s\right)
\left( \psi \left( t\right) -\psi \left( s\right) \right) ^{\alpha -1}g_{z}\left( s\right) ds\right\vert\notag \\
&\leq& \varepsilon I_{a+}^{\alpha ;\psi}1.
\end{eqnarray*}

Hence, we obtain
\begin{eqnarray}\label{casa6}
&&\left\vert z\left( t\right) -\frac{\left( \psi \left( t\right) -\psi \left(a\right) \right) ^{\gamma -1}}{\Gamma \left( \gamma \right) }z_{a}-\frac{1}{\Gamma \left( \alpha \right) }\int_{a}^{t}\psi ^{\prime }\left( s\right)
\left( \psi \left( t\right) -\psi \left( s\right) \right) ^{\alpha-1}g_{z}\left( s\right) ds\right\vert \notag \\
&\leq & \frac{\varepsilon \left( \psi\left( T\right) -\psi \left( a\right) \right) ^{\alpha }}{\Gamma \left(\alpha +1\right) },
\end{eqnarray}
where $g_{z}\in C_{1-\gamma ,\psi }\left( J,\mathbb{R} \right) $ satisfies the functional equation
\begin{equation*}
g_{z}\left( t\right) =f\left( t,\frac{\left( \psi \left( t\right) -\psi \left( a\right) \right) ^{\gamma -1}}{\Gamma \left( \gamma \right) } z_{a}+I_{a+}^{\alpha ;\psi }g_{z}\left( t\right) ,g_{z}\left( t\right)\right) .
\end{equation*}

On the other hand, we have, for each $t\in J$,
\begin{eqnarray}\label{casa7}
&&\left\vert z\left( t\right) -y\left( t\right) \right\vert \notag \\
&\leq &\left\vert z\left( t\right) -\frac{\left( \psi \left( t\right) -\psi \left( a\right) \right) ^{\gamma -1}}{\Gamma \left( \gamma \right) }z_{a}-\frac{1}{\Gamma \left( \alpha \right) }\int_{a}^{t}\psi ^{\prime }\left( s\right) \left( \psi \left( t\right) -\psi \left( s\right) \right) ^{\alpha -1}g_{z}\left( s\right) ds\right\vert  \notag \\ &&+\frac{1}{\Gamma \left( \alpha \right) }\int_{a}^{t}\psi ^{\prime }\left( s\right) \left( \psi \left( t\right) -\psi \left( s\right) \right) ^{\alpha -1}\left\vert g_{z}\left( s\right) -g_{y}\left( s\right) \right\vert ds
\end{eqnarray}
where
\begin{equation*}
g_{y}\left( t\right) =f\left( t,y\left( t\right) ,g_{y}\left( t\right)\right)
\end{equation*}
and
\begin{equation*}
g_{z}\left( t\right) =f\left( t,z\left( t\right) ,g_{z}\left( t\right)\right) .
\end{equation*}

Using hypothesis {\rm (H2)} and the two equalities above for $g_{z}(\cdot)$ and $g_{y}(\cdot)$, we have, for each $t\in J$,
\begin{equation*}
\left\vert g_{z}\left( t\right) -g_{y}\left( t\right) \right\vert \leq k\left\vert z\left( t\right) -y\left( t\right) \right\vert +l\left\vert g_{z}\left( t\right) -g_{y}\left( t\right) \right\vert,
\end{equation*}
which can be written as
\begin{equation}\label{casa8}
\left\vert g_{z}\left( t\right) -g_{y}\left( t\right) \right\vert \leq \frac{k}{1-l}\left\vert z\left( t\right) -y\left( t\right) \right\vert .
\end{equation}

Using Eq.(\ref{casa6}), Eq.(\ref{casa7}), Eq.(\ref{casa8}) and Theorem \ref{teo3}, we have
\begin{eqnarray}\label{casa9}
&&\left\vert z\left( t\right) -y\left( t\right) \right\vert \notag \\
&\leq &\frac{\varepsilon \left( \psi \left( T\right) -\psi \left( a\right) \right) ^{\alpha }}{\Gamma \left( \alpha +1\right) }+\frac{k}{\left( 1-l\right) \Gamma \left( \alpha \right) }\int_{a}^{t}\psi ^{\prime }\left( s\right)
\left( \psi \left( t\right) -\psi \left( s\right) \right) ^{\alpha-1}\left\vert z\left( s\right) -y\left( s\right) \right\vert ds  \notag \\
&\leq &\frac{\varepsilon \left( \psi \left( T\right) -\psi \left( a\right)\right) ^{\alpha }}{\Gamma \left( \alpha +1\right) }\left[ 1+\int_{a}^{t}\underset{n=1}{\overset{\infty }{\sum }}\left( \frac{k}{1-l}\right) ^{n}
\frac{1}{\Gamma \left( n\alpha \right) }\psi ^{\prime }\left( s\right)\left( \psi \left( t\right) -\psi \left( s\right) \right) ^{n\alpha -1}ds\right]  \notag \\
&\leq &\frac{\varepsilon \left( \psi \left( T\right) -\psi \left( a\right)\right) ^{\alpha }}{\Gamma \left( \alpha +1\right) }\left[ 1+\underset{n=1}{\overset{\infty }{\sum }}\frac{1}{\Gamma \left( n\alpha +1\right) }\left[ 
\frac{k}{1-l}\left( \psi \left( T\right) -\psi \left( s\right) \right)^{\alpha }\right] ^{n}\right]  \notag \\
&=&\frac{\varepsilon \left( \psi \left( T\right) -\psi \left( a\right)\right) ^{\alpha }}{\Gamma \left( \alpha +1\right) }\mathbb{E}_{\alpha }\left( \frac{k}{1-l}\left( \psi \left( T\right) -\psi \left( s\right) \right) ^{\alpha
}\right).
\end{eqnarray}

Then, for $c_{f}:=\dfrac{\left( \psi \left( T\right) -\psi \left(
a\right)\right) ^{\alpha }}{\Gamma \left( \alpha +1\right) }\mathbb{E}_{\alpha
}\left( \dfrac{k}{1-l}\left( \psi \left( T\right) -\psi \left( s\right) \right)
^{\alpha}\right) $ with $t\in J=\left[ a,T\right] ,$ $T>a$, we conclude from
Eq.(\ref{casa9}) that Eq.(\ref{CPT}) is Ulam-Hyers stable. On the other hand,
choosing $\Phi \left( \varepsilon \right) =c\varepsilon ,$ $\Phi \left(
0\right) =0,$ we obtain that Eq.(\ref{CPT}) is generalized Ulam-Hyers stable.
\end{proof}

\begin{theorem}\label{teo6} Assume the validity of {\rm (H1)}, {\rm (H2)} and {\rm Eq.(\ref{casa2})}. Assume also the validity of hypothesis {\rm (H3)}, i.e., that function $\varphi \in C\left( J,\mathbb{R}	_{+}\right) $ is increasing and there exists $\lambda _{\varphi }>0$ such that, for each $t\in J,$ we have
\begin{equation*}
I_{a+}^{\alpha ;\psi }\varphi \left( t\right) \leq \lambda _{\varphi
}\varphi \left( t\right) .
\end{equation*}
Then {\rm Eq.(\ref{CPT})} is Ulam-Hyers-Rassias stable with respect to $\varphi$.
\end{theorem}

\begin{proof} Let $z\in C_{1-\gamma ,\psi }\left( J,\mathbb{R} \right) $ be a solution of the inequality
\begin{equation}\label{casa10}
\left\vert ^{H}\mathbb{D}_{a+}^{\alpha ,\beta ;\psi }z\left( t\right) -f\left( t,z\left( t\right), ^{H}\mathbb{D}_{a+}^{\alpha ,\beta ;\psi }z\left( t\right) \right) \right\vert \leq \varepsilon \varphi \left( t\right) ,\text{ }t\in J,\text{ }\varepsilon >0.
\end{equation}

On the other hand, let us denote by $y\in C_{1-\gamma ,\psi }\left( J,\mathbb{R}\right) $ the unique solution of the Cauchy problem
\begin{equation*}
^{H}\mathbb{D}_{a+}^{\alpha ,\beta ;\psi }y\left( t\right) =f\left( t,y\left( t\right), ^{H}\mathbb{D}_{a+}^{\alpha ,\beta ;\psi }y\left( t\right) \right) ,\text{ }t\in J\,,\text{ }0<\alpha \leq 1,\text{ }0\leq \beta \leq 1,
\end{equation*}
\begin{equation*}
I_{a+}^{\alpha ;\psi }y\left( a\right) =I_{a+}^{\alpha ;\psi }z\left( a\right) .
\end{equation*}

Using Lemma \ref{lema1}, we have
\begin{equation*}
y\left( t\right) =\frac{\left( \psi \left( t\right) -\psi \left( a\right) \right) ^{\gamma -1}}{\Gamma \left( \gamma \right) }y_{a}+\frac{1}{\Gamma \left( \alpha \right) }\int_{a}^{t}\psi ^{\prime }\left( s\right) \left( \psi \left( t\right) -\psi \left( s\right) \right) ^{\alpha -1}g_{y}\left( s\right) ds
\end{equation*}
where $g_{z}\in C_{1-\gamma ,\psi }\left( J,\mathbb{R}\right) $ satisfies the functional equation
\begin{equation*}
g_{y}\left( t\right) =f\left( t,\frac{\left( \psi \left( t\right) -\psi \left( a\right) \right) ^{\gamma -1}}{\Gamma \left( \gamma \right) }y_{a}+I_{a+}^{\alpha ;\psi }g_{y}\left( t\right) ,g_{y}\left( t\right) \right) .
\end{equation*}

Applying operator $I_{a+}^{\alpha ;\psi }\left( \cdot \right) $ on both sides of Eq.(\ref{casa10}) and using Theorem \ref{teo1}, we obtain 
\begin{eqnarray}\label{casa11}
&&\left\vert z\left( t\right) -\frac{\left( \psi \left( t\right) -\psi \left( a\right) \right) ^{\gamma -1}}{\Gamma \left( \gamma \right) }z_{a}-I_{a+}^{\alpha ;\psi }f\left( t,z\left( t\right), ^{H}\mathbb{D}_{a+}^{\alpha
,\beta ;\psi }z\left( t\right) \right) \right\vert  \notag \\ &\leq &\varepsilon \frac{1}{\Gamma \left( \alpha \right) }\int_{a}^{t}\psi ^{\prime }\left( s\right) \left( \psi \left( t\right) -\psi \left( a\right) \right) ^{\gamma -1}\varphi \left( s\right) ds\leq \varepsilon \lambda _{\varphi }\varphi \left( t\right),
\end{eqnarray}
where $g_{z}\in C_{1-\gamma ,\psi }\left( J,\mathbb{R}\right) $ satisfies the functional equation
\begin{equation*}
g_{z}\left( t\right) =f\left( t,\frac{\left( \psi \left( t\right) -\psi\left( a\right) \right) ^{\gamma -1}}{\Gamma \left( \gamma \right) }z_{a}+I_{a+}^{\alpha ;\psi }g_{z}\left( t\right) ,g_{z}\left( t\right)\right) .
\end{equation*}

Then, performing the same steps of Eq.(\ref{casa7}), we have, for each $t\in J$,
\begin{eqnarray}\label{casa12}
&&\left\vert z\left( t\right) -y\left( t\right) \right\vert \notag \\
&=&\left\vert z\left( t\right) -\frac{\left( \psi \left( t\right) -\psi \left( a\right) \right) ^{\gamma -1}}{\Gamma \left( \gamma \right) }z_{a}-\frac{1}{\Gamma \left( \alpha \right) }\int_{a}^{t}\psi ^{\prime }\left( s\right) \left( \psi \left( t\right) -\psi \left( s\right) \right) ^{\alpha -1}g_{y}\left( s\right) ds\right\vert  \notag \\
&\leq &\left\vert z\left( t\right) -\frac{\left( \psi \left( t\right) -\psi \left( a\right) \right) ^{\gamma -1}}{\Gamma \left( \gamma \right) }z_{a}-\frac{1}{\Gamma \left( \alpha \right) }\int_{a}^{t}\psi ^{\prime }\left(
s\right) \left( \psi \left( t\right) -\psi \left( s\right) \right) ^{\alpha -1}g_{z}\left( s\right) ds\right\vert  \notag \\
&&+\frac{1}{\Gamma \left( \alpha \right) }\int_{a}^{t}\psi ^{\prime }\left( s\right) \left( \psi \left( t\right) -\psi \left( s\right) \right) ^{\alpha -1}\left\vert g_{z}\left( s\right) -g_{y}\left( s\right) \right\vert ds
\end{eqnarray}
where
\begin{equation*}
g_{y}\left( t\right) =f\left( t,y\left( t\right) ,g_{y}\left( t\right)\right)
\end{equation*}
and
\begin{equation*}
g_{z}\left( t\right) =f\left( t,z\left( t\right) ,g_{z}\left( t\right)\right) .
\end{equation*}

Then, using Eq.(\ref{casa8}), Eq.(\ref{casa11}) and Eq.(\ref{casa12}), we have
\begin{eqnarray}\label{casa13}
&&\left\vert z\left( t\right) -y\left( t\right) \right\vert \notag \\
&\leq &\varepsilon \lambda _{\varphi }\varphi \left( t\right) +\frac{k}{\left( 1-l\right) \Gamma \left( \alpha \right) }\int_{a}^{t}\psi ^{\prime }\left( s\right)
\left( \psi \left( t\right) -\psi \left( s\right) \right) ^{\alpha -1}\left\vert z\left( s\right) -y\left( s\right) \right\vert ds  \notag \\
&\leq &\varepsilon \lambda _{\varphi }\varphi \left( t\right) +\frac{k\left\Vert z-y\right\Vert _{\infty }}{\left( 1-l\right) \Gamma \left( \alpha \right) }\int_{a}^{t}\psi ^{\prime }\left( s\right) \left( \psi \left( t\right) -\psi \left( s\right) \right) ^{\alpha -1}ds  \notag \\
&\leq &\varepsilon \lambda _{\varphi }\varphi \left( t\right) +\frac{ k\left\Vert z-y\right\Vert _{\infty }}{\left( 1-l\right) \Gamma \left( \alpha +1\right) }\left( \psi \left( T\right) -\psi \left( s\right) \right) ^{\alpha }.
\end{eqnarray}

Thus, we have
\begin{equation}\label{casa14}
\left\Vert z-y\right\Vert _{\infty }\left[ 1-\frac{k\left( \psi \left(t\right) -\psi \left( a\right) \right) ^{\alpha }}{\left( 1-l\right) \Gamma\left( \alpha +1\right) }\right] \leq \varepsilon \lambda _{\varphi }\varphi\left( t\right) .
\end{equation}

Using Eq.(\ref{casa2}), we can rewrite Eq.(\ref{casa14}) in the following form 
\begin{equation*}
\left\Vert z-y\right\Vert _{\infty }\leq \left[ 1-\frac{k\left( \psi \left( t\right) -\psi \left( a\right) \right) ^{\alpha }}{\left( 1-l\right) \Gamma \left( \alpha +1\right) }\right] ^{-1}\varepsilon \lambda _{\varphi }\varphi
\left( t\right) .
\end{equation*}

Then, for $c_{f}:=\left[ 1-\dfrac{k\left( \psi \left( t\right) -\psi \left( a\right) \right) ^{\alpha }}{\left( 1-l\right) \Gamma \left( \alpha +1\right) }\right] ^{-1}\lambda _{\varphi }$ with $t\in J=\left[ a,T\right]$, $T>a,$ we conclude from Eq.(\ref{casa13}) that Eq.(1) is Ulam-Hyers-Rassias stable.  
\end{proof}

Proposing a fractional operator of differentiation or integration is important and motivating to researchers working in the field of fractional calculus. However, proposing a fractional operator that unifies a large number of definitions is not a simple task, let alone an easy one. The $\psi$-Hilfer operator defined in Eq.(\ref{A}) contains, as special cases,  several different classes (definitions) of fractional derivatives. Thus, the nonlinear Cauchy problem we have just proposed does also contain, as special cases, corresponding Cauchy problems for those classes of fractional derivatives. Moreover, the results we have just proved about the uniqueness of its solution and its stabilities also apply to the special cases.

These classes of fractional derivatives are obtained by means of an adequate choice of $\psi(\cdot)$ and by considering the limits $\beta\rightarrow 1$ or $\beta \rightarrow 0$.

The choice of parameter $a$ also determines the particular cases of the operators. For instance, in order to recover our results for the Hadamard fractional derivative, we choose $\psi(t)=\ln t$, take the limit $\beta
\rightarrow 0 $  and put $a=1$, because $\ln t$ is not defined for $a=0$. 

We now present some of these  particular cases of Ulam-Hyers and Ulam-Hyers-Rassias stabilities.  In the next section, we present some examples and discuss some particular derived from them. 

\begin{itemize}

\item Choosing $\psi(t)=\ln t$, $a=1$ and applying the limit $\beta \rightarrow 0$ on both sides of Eq.(\ref{CPT}) and applying Theorem \ref{teo5} and Theorem \ref{teo6}, we obtain the stability of Ulam-Hyers and Ulam-Hyers-Rassias for the Cauchy problem with Hadamard fractional derivative \cite{PRIN};

\item For $\psi(t)=t$, taking the limit $\beta \rightarrow 1$ on both sides of Eq.(\ref{CPT}) and applying Theorem \ref{teo5} and Theorem \ref{teo6}, we obtain the stability of Ulam-Hyers and	Ulam-Hyers-Rassias for the Cauchy problem with the Caputo fractional derivative \cite{PRIN1};

\item For $\psi(t)=t$, taking the limit $\beta \rightarrow 0$ on both sides of Eq.(\ref{CPT}) and applying Theorem \ref{teo5} and Theorem \ref{teo6}, we obtain the stability of Ulam-Hyers and	Ulam-Hyers-Rassias for the Cauchy problem with the Riemann-Liouville fractional derivative;

\item Choosing $\psi(t)=t ^{\rho}$, taking the limit $\beta \rightarrow 0 $ on both sides of Eq.(\ref{CPT}) and applying Theorem \ref{teo5} and	Theorem \ref{teo6}, we obtain the stability of Ulam-Hyers and
Ulam-Hyers-Rassias for the Cauchy problem with the Katugampola fractional derivative;

\item For $\psi(t)=\ln t$, $a=1$ and applying the limit $\beta \rightarrow 0$ on both sides of Eq.(\ref{CPT}) and applying Theorem \ref{teo5} and Theorem \ref{teo6}, we obtain the stability of Ulam-Hyers and
Ulam-Hyers-Rassias for the Cauchy problem with the Caputo-Hadamard fractional derivative.

\end{itemize}

We thus see that the Cauchy problem proposed for the $\psi$-Hilfer fractional derivative is in fact general, and with this it is possible to deduce the existence and uniqueness, as well as the stability of
Ulam-Hyers and Ulam-Hyers-Rassias, for the non-linear Cauchy problem Eq.(\ref{CPT}). 


\section{Examples}

In this section, we consider some particular Cauchy problems in order to apply our results to study their generalized Hyers-Ulam and Hyers-Ulam-Rassias stabilities.

Consider the following nonlinear Cauchy problem: 
\begin{equation}\label{casa15}
\left\{ 
\begin{array}{rcl}
^{H}\mathbb{D}_{a+}^{\alpha ,\beta ;\psi }y\left( t\right) & = & \dfrac{\lambda }{20} \mathbb{E}_{\alpha }\left( \lambda ^{\beta }\left( \psi \left( t\right) -\psi \left( a\right) \right) ^{\alpha }\right) y\left( t\right) +\dfrac{\lambda }{10}\text{ }^{H}\mathbb{D}_{a+}^{\alpha ,\beta ;\psi }y\left( t\right) ,\text{ }t\in \left[ a,T\right] \\ I_{a+}^{\alpha ;\psi }y\left( a\right) & = & 1
\end{array}
\right.
\end{equation}
where $\mathbb{E}_{\alpha}(\cdot)$ is the Mittag-Leffler, $^{H}\mathbb{D}^{\alpha,\beta;\psi}(\cdot)$ is the $\psi$-Hilfer fractional derivative and $I^{\alpha;\psi}_{a+}(\cdot)$ is $\psi$-Riemann-Liouville fractional integral.

The following examples are particular cases of the Cauchy problem given by Eq.(\ref{casa15}).

\begin{example}

Taking $\left( t\right) =t,$ $a=0,$ $T=1,$ $\alpha =1/2 $ and $\beta\rightarrow 0$, we get a particular case of the Cauchy problem {\rm Eq.(\ref{casa15})} involving the Riemann-Liouville fractional derivative, given by

\begin{equation}\label{casa16}
\left\{ 
\begin{array}{rcl}
^{RL}\mathcal{D}_{0+}^{1/2,0;t}y\left( t\right) & = & \dfrac{\lambda }{20}\mathbb{E}_{1/2}\left( t^{1/2}\right) y\left( t\right) +\dfrac{\lambda }{10}\text{ }^{RL}\mathcal{D}_{0+}^{1/2,0;t}y\left( t\right) ,\text{ }t\in \left[ 0,1\right]  \\  I_{0+}^{1/2;t}y\left( 0\right) & = & 1
\end{array}
\right.
\end{equation}
with $0<\lambda <\dfrac{13}{5}$ and
\begin{equation*}
f\left( t,u,v\right) =\frac{\lambda }{20}\mathbb{E}_{1/2}\left( t^{1/2}\right) u+\frac{\lambda }{10}v,\text{ }t\in \left[ 0,1\right] ,\text{ }u,v\in \mathbb{R}.
\end{equation*}

For all $u,v,\overline{u},\overline{v}\in \mathbb{R}$ and $t\in \left[ 0,1\right]$, we have
\begin{eqnarray*}
\left\vert f\left( t,u,v\right) -f\left( t,\overline{u},\overline{v}\right) \right\vert &=&\left\vert \frac{\lambda }{20}\mathbb{E}_{1/2}\left( t^{1/2}\right) u+\frac{\lambda }{10}v-\frac{\lambda }{20}\mathbb{E}_{1/2}\left( t^{1/2}\right)  
\overline{u}-\frac{\lambda }{10}\overline{v}\right\vert  \notag \\
&\leq &\frac{\lambda }{20}\mathbb{E}_{1/2}\left( t^{1/2}\right) \left\vert u-\overline{u}\right\vert +\frac{\lambda }{10}\left\vert v-\overline{v}\right\vert.
\end{eqnarray*}

Thus, condition {\rm (H2)} is satisfied with $k=\dfrac{\lambda }{20} \mathbb{E}_{1/2}\left( t^{1/2}\right) $ and $l=\dfrac{\lambda }{10}$. On the other hand, using {\rm Eq.(\ref{A7})}, the condition
\begin{equation*}
\frac{k\left( \psi \left( t\right) -\psi \left( a\right) \right) ^{\alpha }}{\Gamma \left( \alpha +1\right) \left( 1-l\right) }=\frac{\frac{\lambda }{20}\mathbb{E}_{1/2}\left( t^{1/2}\right) t^{1/2}}{\Gamma \left( 3/2\right) \left( 1-
\frac{\lambda }{10}\right) }=\frac{\lambda \mathbb{E}_{1/2}\left( t^{1/2}\right) t^{1/2}}{\sqrt{\pi }\left( 10-\lambda \right) }\leq \frac{\lambda \mathbb{E}_{1/2}\left( 1\right) }{\sqrt{\pi }\left( 10-\lambda \right) }<1,
\end{equation*}
for $0<\lambda <\dfrac{13}{5}$, is satisfied.

So, by {\rm Theorem \ref{teo4}}, the problem has a unique solution.	Consequently, by {\rm Theorem \ref{teo5}}, {\rm Eq.(\ref{casa16})} has Ulam-Hyers stability.

\end{example}

\begin{example} For $\psi \left( t\right) =t,$ $a=0,$ $T=1,$ $\alpha =1/2$ and 	$\beta \rightarrow 1$, we obtain a particular case of the Cauchy problem {\rm Eq.(\ref{casa15})} using the Caputo fractional derivative: 
\begin{equation}\label{casa17}
\left\{ 
\begin{array}{rcl}
^{C}\mathscr{D}_{0+}^{1/2,1;t}y\left( t\right) & = & \dfrac{\lambda }{20}\mathbb{E}_{1/2}\left( \lambda t^{1/2}\right) y\left( t\right) +\dfrac{\lambda }{10} \text{ }^{C}\mathscr{D}_{0+}^{1/2,1;t}y\left( t\right) ,\text{ }t\in \left[ 0,1\right] \\ I_{0+}^{1;t}y\left( 0\right) & = & 1
\end{array}
\right.
\end{equation}
with $0<\lambda <\dfrac{13}{10}$ and
\begin{equation*}
f\left( t,u,v\right) =\frac{\lambda }{20}\mathbb{E}_{1/2}\left( \lambda t^{1/2}\right) u+\frac{\lambda }{10}v,\text{ }t\in \left[ 0,1\right] ,\text{ }u,v\in \mathbb{R}.
\end{equation*}

For all $u,v,\overline{u},\overline{v}\in \mathbb{R}$ and $t\in \left[ 0,1\right]$, we have
\begin{eqnarray*}
\left\vert f\left( t,u,v\right) -f\left( t,\overline{u},\overline{v}\right) \right\vert &=&\left\vert \frac{\lambda }{20}\mathbb{E}_{1/2}\left( \lambda t^{1/2}\right) u+\frac{\lambda }{10}v-\frac{\lambda }{20}\mathbb{E}_{1/2}\left(\lambda t^{1/2}\right) \overline{u}-\frac{\lambda }{10}\overline{v}\right\vert  \notag \\
&\leq &\frac{\lambda }{20}\mathbb{E}_{1/2}\left( \lambda t^{1/2}\right) \left\vert u-\overline{u}\right\vert +\frac{\lambda }{10}\left\vert v-\overline{v}\right\vert .
\end{eqnarray*}

Thus, condition {\rm (H2)} is satisfied with $k=\dfrac{\lambda}{20}\mathbb{E}_{1/2}\left( \lambda t^{1/2}\right) $ and $l=\dfrac{\lambda }{10}$. On the other hand, the condition
\begin{equation*}
\frac{k\left( \psi \left( t\right) -\psi \left( a\right) \right) ^{\alpha }}{ \Gamma \left( \alpha +1\right) \left( 1-l\right) }=\frac{\lambda \mathbb{E}_{1/2}\left( \lambda t^{1/2}\right) t^{1/2}}{\sqrt{\pi }\left( 10-\lambda
\right) }\leq \frac{\lambda \mathbb{E}_{1/2}\left( \lambda \right) }{\sqrt{\pi } \left( 10-\lambda \right) }<1,
\end{equation*}
for $0<\lambda <\dfrac{13}{10}$, is satisfied.

So, by {\rm Theorem \ref{teo4}}, the problem has a unique solution.	Consequently, by {\rm Theorem \ref{teo5}}, {\rm Eq.(\ref{casa17})} has Ulam-Hyers stability.

\end{example}

\begin{example} For $\psi \left( t\right) =\ln t,$ $a=1,$ $T=e,$ $\alpha =1/2$
	and $\beta \rightarrow 0,$ we have a particular case of the Cauchy
	problem {\rm Eq.(\ref{casa15})} involving the Hadamard fractional
	derivative:
\begin{equation}\label{casa18}
\left\{ 
\begin{array}{rcl}
^{HD}\mathscr{D}_{1+}^{1/2,0;\ln t}y\left( t\right) & = & \dfrac{\lambda }{20}\mathbb{E}_{1/2}\left( \ln t^{1/2}\right) y\left( t\right) +\dfrac{\lambda }{10}\text{ }^{HD}\mathscr{D}_{1+}^{1/2,0;\ln t}y\left( t\right) ,\text{ }t\in \left[ 1,e\right] \\ I_{1+}^{1/2;\ln t}y\left( 1\right) & = & 1
\end{array}
\right.
\end{equation}
with $0<\lambda <\dfrac{13}{5}$ and
\begin{equation*}
f\left( t,u,v\right) =\frac{\lambda }{20}\mathbb{E}_{1/2}\left( \ln t^{1/2}\right) u+\frac{\lambda }{10}v,\text{ }t\in \left[ 1,e\right] ,\text{ }u,v\in \mathbb{R}.
\end{equation*}

For all $u,v,\overline{u},\overline{v}\in \mathbb{R}$ and $t\in \left[ 1,e\right]$, we have
\begin{eqnarray*}
\left\vert f\left( t,u,v\right) -f\left( t,\overline{u},\overline{v}\right) \right\vert &=&\left\vert \frac{\lambda }{20}\mathbb{E}_{1/2}\left( \ln t^{1/2}\right) u+\frac{\lambda }{10}v-\frac{\lambda }{20}\mathbb{E}_{1/2}\left( \ln
t^{1/2}\right) \overline{u}-\frac{\lambda }{10}\overline{v}\right\vert  \notag \\
&\leq &\frac{\lambda }{20}\mathbb{E}_{1/2}\left( \ln t^{1/2}\right) \left\vert u-\overline{u}\right\vert +\frac{\lambda }{10}\left\vert v-\overline{v}\right\vert.
\end{eqnarray*}

Thus, condition {\rm (H2)} is satisfied with $k=\dfrac{\lambda }{20}\mathbb{E}_{1/2}\left( \ln t^{1/2}\right) $ and $l=\dfrac{\lambda }{10}$. On the other hand, using {\rm Eq.(\ref{A7})}, the condition
\begin{equation*}
\frac{k\left( \psi \left( t\right) -\psi \left( a\right) \right) ^{\alpha }}{\Gamma \left( \alpha +1\right) \left( 1-l\right) }=\frac{\frac{\lambda }{20}\mathbb{E}_{1/2}\left( \ln t^{1/2}\right) \left( \ln t-\ln a\right) }{\Gamma \left(3/2\right) \left( 1-\frac{\lambda }{10}\right) }\leq \frac{\lambda \mathbb{E}_{1/2}\left( 1\right) }{\sqrt{\pi }\left( 10-\lambda \right) }<1,
\end{equation*}
for $0<\lambda <\dfrac{13}{5}$, is satisfied.

So, by {\rm Theorem \ref{teo4}}, the problem has a unique solution.	Consequently, by {\rm Theorem \ref{teo5}} the {\rm Eq.(\ref{casa18})}, has Ulam-Hyers stability.

\end{example}

\begin{example} For $\psi \left( t\right) =t,$ $a=0,$ $T=1,$ $\alpha =1/2$ and $\beta =1/2,$ we obtain a particular case of the Cauchy problem {\rm	Eq.(\ref{casa15})} involving the $\psi$-Hilfer fractional derivative: 
\begin{equation}\label{casa19}
\left\{ 
\begin{array}{rcl}
^{H}\mathbb{D}_{0+}^{1/2,1/2;t}y\left( t\right) & = & \dfrac{\lambda }{20}\mathbb{E}_{1/2}\left( \lambda ^{1/2}t^{1/2}\right) y\left( t\right) +\dfrac{\lambda }{10}\text{ }^{H}\mathbb{D}_{0+}^{1/2,1/2;t}y\left( t\right) ,\text{ }t\in \left[ 0,1\right] \\ I_{0+}^{3/4;t}y\left( 1\right) & = & 1
\end{array}
\right.
\end{equation}
with $0<\lambda <\dfrac{13}{10}$ and
\begin{equation*}
f\left( t,u,v\right) =\frac{\lambda }{20}\mathbb{E}_{1/2}\left( \lambda^{1/2}t^{1/2}\right) u+\frac{\lambda }{10}v,\text{ }t\in \left[ 0,1\right],\text{ }u,v\in \mathbb{R}.
\end{equation*}

For all $u,v,\overline{u},\overline{v}\in \mathbb{R}$ and $t\in \left[ 0,1\right]$, we get
\begin{eqnarray*}
\left\vert f\left( t,u,v\right) -f\left( t,\overline{u},\overline{v}\right) \right\vert &=&\left\vert \frac{\lambda }{20}\mathbb{E}_{1/2}\left( \lambda ^{1/2}t^{1/2}\right) u+\frac{\lambda }{10}v-\frac{\lambda }{20}\mathbb{E}_{1/2}\left(\lambda ^{1/2}t^{1/2}\right) \overline{u}-\frac{\lambda }{10}\overline{v}\right\vert  \notag \\ &\leq &\frac{\lambda }{20}\mathbb{E}_{1/2}\left( \lambda ^{1/2}t^{1/2}\right)\left\vert u-\overline{u}\right\vert +\frac{\lambda }{10}\left\vert v-\overline{v}\right\vert .
\end{eqnarray*}

Thus, condition {\rm (H2)} is satisfied with $k=\dfrac{\lambda }{20}\mathbb{E}_{1/2}\left( \lambda ^{1/2}t^{1/2}\right) $ and $l=\dfrac{\lambda }{10}$. On the other hand, the condition
\begin{equation*}
\frac{k\left( \psi \left( t\right) -\psi \left( a\right) \right) ^{\alpha }}{\Gamma \left( \alpha +1\right) \left( 1-l\right) }=\frac{\frac{\lambda }{20}\mathbb{E}_{1/2}\left( \lambda ^{1/2}t^{1/2}\right) t^{1/2}}{\Gamma \left( 3/2\right) \left( 1-\frac{\lambda }{10}\right) }\leq \frac{\lambda \mathbb{E}_{1/2}\left(\lambda ^{1/2}\right) }{\sqrt{\pi }\left( 10-\lambda \right) }<1,
\end{equation*}
for $0<\lambda <\dfrac{13}{10}$, is satisfied.

So, by {\rm Theorem \ref{teo4}}, the problem has a unique solution. Consequently, by {\rm Theorem \ref{teo5}}, the {\rm Eq.(\ref{casa19})} has Ulam-Hyers stability.

\end{example}

\begin{remark} We presented four particular cases for the study of Ulam-Hyers stability. We have seen that for the Cauchy problems using the Riemann-Liouville and Hadamard fractional derivatives, the stability range $0<\lambda <\dfrac{13}{5}$ is the same. On the other hand, for the Cauchy problem involving the Caputo fractional derivative, the interval is $0<\lambda<\dfrac {13} {10}$.  Thus, we can say that the stability interval for parameter $\lambda$ is larger when we use Riemann-Liouville and Hadamard derivatives than when we considered the Caputo derivative. Therefore, we can conclude that, for the Cauchy problem {\rm Eq.\ref{casa15})}, its particular cases have different	stability intervals for each particular fractional derivative type. 

\end{remark}

The following examples are directed to the stability of Ulam-Hyers-Rassias.

Consider the following nonlinear Cauchy problem: 
\begin{equation}\label{casa20}
\left\{ 
\begin{array}{rcl}
^{H}\mathbb{D}_{a+}^{\alpha ,\beta ;\psi }y\left( t\right) & = & \dfrac{\lambda ^{\beta }}{20}\left( \psi \left( t\right) -\psi \left( a\right) \right) ^{\alpha }\cos t\text{ }y\left( t\right) +\dfrac{\lambda ^{\beta }}{20}
\text{ }^{H}\mathbb{D}_{a+}^{\alpha ,\beta ;\psi }y\left( t\right) ,\text{ }t\in \left[ a,T\right] \\ 
I_{a+}^{\alpha ;\psi }y\left( a\right) & = & 1
\end{array}
\right.
\end{equation}

The following examples are particular cases of the Cauchy problem given by
Eq.(\ref{casa20}).

\begin{example} Taking $\psi \left( t\right) =\ln t,$ $a=1,$ $T=e,$ $\alpha =1/2$ and $\beta \rightarrow 0,$ we get a particular case of the Cauchy	problem {\rm Eq.(\ref{casa20})} using the Hadamard fractional derivative: 

\begin{equation}\label{casa21}
\left\{ 
\begin{array}{rcl}
^{HD}\mathscr{D}_{1+}^{1/2,0;\ln t}y\left( t\right)  & = & \dfrac{1}{20}\ln t^{1/2}\cos t\text{ }y\left( t\right) +\dfrac{1}{20}\text{ }^{HD}\mathscr{D}_{1+}^{1/2,0;\ln t}y\left(t\right) ,\text{ }t\in \left[ 1,e\right]  \\I_{1+}^{1/2;\ln t}y\left( 1\right)  & = & 1
\end{array}
\right. 
\end{equation}
and
\begin{equation*}
f\left( t,u,v\right) =\frac{\ln t^{1/2}}{20}\cos t\text{ }u+\frac{v}{20},\text{ }u,v\in \mathbb{R},\text{ }t\in \left[ 1,e\right].
\end{equation*}

Note that $f$ is a set of continuous functions. Then, for all $u,v,\overline{u},\overline{v}\in \mathbb{R}$ and $t\in \left[ 1,e\right]$, we have
\begin{eqnarray*}
\left\vert f\left( t,u,v\right) -f\left( t,\overline{u},\overline{v}\right)\right\vert &=&\left\vert \frac{\ln t^{1/2}}{20}\cos t\text{ }u+\frac{v}{20}-\frac{\ln t^{1/2}}{20}\cos t\text{ }\overline{u}-\frac{\overline{v}}{20}
\right\vert  \notag \\&\leq &\frac{1}{20}\left( \left\vert u-\overline{u}\right\vert +\left\vert v-\overline{v}\right\vert \right) .
\end{eqnarray*}

Thus, condition {\rm (H2)} is satisfied with $k=l=1/20$. So, for $\varphi \left( t\right) =\ln t^{1/2},$ we get
\begin{eqnarray*}
I_{1+}^{1/2;\ln t}\varphi \left( t\right) &=&\frac{1}{\Gamma \left( 1/2\right) }\int_{1}^{t}\left( \ln \frac{t}{s}\right) ^{-1/2}\left( \ln s\right) ^{1/2}\frac{ds}{s}  \notag \\ &\leq &\frac{1}{\Gamma \left( 1/2\right) }\int_{1}^{t}\left( \ln \frac{t}{s}\right) ^{-1/2}\frac{ds}{s}  \notag \\ &=&\frac{2}{\sqrt{\pi }}\left( \ln t\right) ^{1/2}.
\end{eqnarray*}

Then, for $\lambda _{\varphi }:=\dfrac{2}{\sqrt{\pi }}$ and $\varphi \left(	t\right) =\left( \ln t\right) ^{1/2},$ condition {\rm (H3)} is satisfied. So, by {\rm Theorem \ref{teo4}} the problem has a unique
solution in $J$. Consequently, from {\rm Theorem \ref{teo6}}, {\rm	Eq.(\ref{casa21})} has Ulam-Hyers-Rassias stability.

\end{example}

\begin{example}
Taking $\psi \left( t\right) =\ln t,$ $a=1,$ $T=e$, $\alpha=1/2$ and $\beta \rightarrow 1$, we obtain a particular case of the Cauchy problem {\rm Eq.(\ref{casa20})} involving the Caputo-Hadamard fractional derivative: 
\begin{equation}\label{casa22}
\left\{ 
\begin{array}{rcl}
^{CH}\mathscr{D}_{1+}^{1/2,1;\ln t}y\left( t\right)  & = & \dfrac{\lambda }{20}\ln t^{1/2}\cos t\text{ }y\left( t\right) +\dfrac{\lambda }{20}\text{ }^{CH}\mathscr{D}_{1+}^{1/2,1;\ln t}y\left( t\right) ,\text{ }t\in \left[ 1,e\right] \\ I_{1+}^{1/2;\ln t}y\left( 1\right)  & = & 1
\end{array}
\right. 
\end{equation}
and
\begin{equation*}
f\left( t,u,v\right) =\frac{\lambda \ln t^{1/2}}{20}\cos t\text{ }u+\frac{\lambda v}{20},\text{ }u,v\in \mathbb{R}
,\text{ }t\in \left[ 1,e\right].
\end{equation*}

Note that $f$ is a set of continuous functions. Then, for all $u,v,\overline{u},\overline{v}\in \mathbb{R}
$ and $t\in \left[ 1,e\right]$, we get
\begin{eqnarray*}
\left\vert f\left( t,u,v\right) -f\left( t,\overline{u},\overline{v}\right) \right\vert  &=&\left\vert \frac{\lambda \ln t^{1/2}}{20}\cos t\text{ }u+\frac{\lambda v}{20}-\frac{\lambda \ln t^{1/2}}{20}\cos t\text{ }\overline{u}
-\frac{\lambda \overline{v}}{20}\right\vert   \notag \\ &\leq &\frac{\lambda }{20}\left( \left\vert u-\overline{u}\right\vert+\left\vert v-\overline{v}\right\vert \right) .
\end{eqnarray*}

Then,  condition {\rm (H2)} is satisfied with $k=l=\lambda /20$. So, for $\varphi \left( t\right) =\ln t^{1/2},$ we have
\begin{eqnarray*}
I_{1+}^{1/2;\ln t}\varphi \left( t\right) &=&\frac{1}{\Gamma \left( 1/2\right) }\int_{1}^{t}\left( \ln \frac{t}{s}\right) ^{-1/2}\left( \ln s\right) ^{1/2}\frac{ds}{s}  \notag \\ &\leq &\frac{1}{\Gamma \left( 1/2\right) }\int_{1}^{t}\left( \ln \frac{t}{s}\right) ^{-1/2}\frac{ds}{s}  \notag \\
&=&\frac{2}{\sqrt{\pi }}\left( \ln t\right) ^{1/2}.
\end{eqnarray*}

Then, for $\lambda _{\varphi }:=\dfrac{2}{\sqrt{\pi }}$ and $\varphi\left(t\right) =\left( \ln t\right) ^{1/2},$  condition {\rm (H3)} is satisfied. So, by {\rm Theorem \ref{teo4}} the problem has a unique
solution in $J$. Consequently, from {\rm Theorem \ref{teo6}}, {\rm	Eq.(\ref{casa22})} has Ulam-Hyers-Rassias stability.

\end{example}

\begin{example} For $\psi \left( t\right) =t^{\rho }$, $0\leq \rho \leq 4$, $a=0,$ $T=1,$ $\alpha =1/2$ and $\beta \rightarrow 1,$ we get a	particular case of the Cauchy problem {\rm Eq.(\ref{casa20})} using the Caputo-Katugampola fractional derivative:
\begin{equation}\label{casa23}
\left\{ 
\begin{array}{rcl}
^{CK}D_{0+}^{1/2,1;t^{\rho }}y\left( t\right)  & = & \dfrac{\lambda }{20}t^{\rho /2}\cos t\text{ }y\left( t\right) +\dfrac{\lambda }{20}\text{ }^{CK}D_{0+}^{1/2,1;t^{\rho }}y\left( t\right) ,\text{ }t\in \left[ 0,1\right]\\ 
I_{0+}^{1/2;t^{\rho }}y\left( 0\right)  & = & 1\end{array}\right. 
\end{equation}
and
\begin{equation*}
f\left( t,u,v\right) =\frac{\lambda t^{\rho /2}}{20}\cos t\text{ }u+\frac{\lambda v}{20},\text{ }u,v\in \mathbb{R}
,\text{ }t\in \left[ 0,1\right].
\end{equation*}

Note that $f$ is a set of continuous functions. Then, for all $u,v,\overline{u},\overline{v}\in \mathbb{R}
$ and $t\in \left[ 0,1\right]$, we have
\begin{eqnarray*}
\left\vert f\left( t,u,v\right) -f\left( t,\overline{u},\overline{v}\right)\right\vert  &=&\left\vert \frac{\lambda t^{\rho /2}}{20}\cos t\text{ }u+\frac{\lambda v}{20}-\frac{\lambda t^{\rho /2}}{20}\cos t\text{ }\overline{u}
-\frac{\lambda \overline{v}}{20}\right\vert   \notag \\ &\leq &\frac{\lambda }{20}\left( \left\vert u-\overline{u}\right\vert +\left\vert v-\overline{v}\right\vert \right) .
\end{eqnarray*}

Then,  condition {\rm (H2)} is satisfied with $k=l=\lambda /20$. So, for $\varphi \left( t\right) =t^{\rho /2}$, we get
\begin{eqnarray*}
^{\rho }I_{0+}^{1/2;t^{\rho }}\varphi \left( t\right)  &=&\frac{\rho ^{1/2}}{\Gamma \left( 1/2\right) }\int_{0}^{t}\left( t^{\rho }-s^{\rho }\right) ^{-1/2}s^{\rho +1/2}ds  \notag \\ &\leq &\frac{\rho ^{1/2}}{\Gamma \left( 1/2\right) }\int_{0}^{t}\left( t^{\rho }-s^{\rho }\right) ^{-1/2}ds  \notag \\ &=&\frac{2\rho ^{1/2}}{\sqrt{\pi }\left( \rho +1\right) }t^{\rho /2}.
\end{eqnarray*}

Then, for $\lambda _{\varphi }:=\dfrac{2\rho ^{1/2}}{\sqrt{\pi }\left( \rho +1\right) }$ and $\varphi \left( t\right) =t^{\rho /2}$,  condition	{\rm (H3)} is satisfied. So, by {\rm Theorem \ref{teo4}}, the problem
has a unique solution in $J$. Consequently, from {\rm Theorem \ref{teo6}},  {\rm Eq.(\ref{casa20})} has Ulam-Hyers-Rassias stability.

\end{example}


\section{Concluding remarks}

In the first part of this paper we studied the existence and uniqueness of solution of a nonlinear Cauchy problem. In the second part we investigated the Ulam-Hyers and Ulam-Hyers-Rassias stabilities of such solutions and discussed some particular cases. Some important points were presented in the body of the article, in particular, the big gain obtained by studying the Cauchy problem using the $\psi$-Hilfer fractional derivative. We discussed some examples of stability, presenting the stability interval of each Cauchy problem and
comparing them. 

An interesting extension of our studies would be to discuss global attractivity for the nonlinear Cauchy problem using the $\psi$-Hilfer fractional derivative \cite{FUT,FUT1,FUT2}. This topic will be the subject of a forthcoming paper \cite{oi}.

\bibliography{ref}
\bibliographystyle{plain}

\end{document}